\documentclass[preprint,3p,times]{elsarticle}
\usepackage{amsmath,amssymb}
\usepackage[all]{xy}
\usepackage{latexsym}
\usepackage{amsthm,a4wide,color}
\usepackage{amsmath,amscd,verbatim}
\usepackage{hyperref}
\usepackage{mathrsfs}
\usepackage{graphicx}

\input diagxy
\newcommand\thda{\mathrel{\rotatebox[origin=c]{-90}{$\twoheadrightarrow$}}}
\newcommand\thua{\mathrel{\rotatebox[origin=c]{90}{$\right) $}}}
\theoremstyle{plain}
\newtheorem{thm}{Theorem}[section]
\newtheorem{lem}[thm]{Lemma}
\newtheorem{prop}[thm]{Proposition}
\newtheorem{cor}[thm]{Corollary}
\theoremstyle{definition}
\newtheorem{defn}[thm]{Definition}
\newtheorem{exmp}[thm]{Example}
\newtheorem{rem}[thm]{Remark}
\begin{document}
\newcommand{\Ran}{\mathrm{Ran}}
\newcommand{\id}{\mathrm{id}}
\newcommand{\p}{{\preccurlyeq}}
\newcommand{\pp}{{\preceq}}
\newcommand{\sq}{{\sqsubseteq}}
\newcommand{\va}{{\vartriangle}}
\newcommand{\ua}{{\uparrow}}
\newcommand{\da}{{\downarrow}}
\newcommand{\la}{{\lambda}}

\numberwithin{equation}{section}

\renewcommand{\theequation}{\thesection.\arabic{equation}}

\begin{frontmatter}

\title{Characterization of  t-norms  on normal convex functions}

\author{Jie Sun\corref{cor}}
\ead{jiesun1027@163.com}

\cortext[cor]{Corresponding author.}
\address{School of Mathematics,  Southwest Minzu University, Chengdu 610041, China}

\begin{abstract}
Type-2  fuzzy set (T2 FS) were introduced by Zadeh in 1965,  and the membership degrees of  T2 FSs are type-1 fuzzy sets (T1 FSs).  Owing to the fuzziness of membership degrees,   T2 FSs can better model the uncertainty of real life, and thus, type-2 rule-based  fuzzy systems (T2 RFSs) become  hot research topics in recent decades. 
  In T2 RFS, the compositional rule of inference is based on triangular norms (t-norms) defined   on complete lattice  \((\mathbf{L},\sq)\) 
  (\(\mathbf{L}\) is the set of all convex normal functions from \([0,1]\) to \([0,1]\), and \(\sq\) is the so-called convolution order). Hence, the  choice of t-norm  on \((\mathbf{L},\sq)\) may influence the performance  of T2 RFS. Therefore, it is significant to broad the set of t-norms among which domain  experts can choose most suitable one.  To construct t-norms on \((\mathbf{L},\sq)\),  the mainstream method is convolution  \(\ast_\va\) which is induced by two operators \(\ast\) and \(\va\) on the unit interval \([0,1]\).  A key problem appears naturally, when convolution  \(\ast_\va\) is a t-norm on  \((\mathbf{L},\sq)\). This paper has solved this problem completely. Moreover,  note that the computational complexity of operators prevent the application of T2 RFSs.   This paper also provides one kind of convolutions which are t-norms on \((\mathbf{L},\sq)\) and extremely easy to calculate.\end{abstract}

\end{frontmatter}

\section{Introduction}
 Fuzzy set theory were introduced by Zadeh in 1965 \cite{za65}.  Later on,  one rule-based  fuzzy system (RFS) were given by Mamdani and Assilian \cite{ma74, ma75} to control a steam engine, that is  the  earliest   application of fuzzy sets in real world. Until now RFSs have been successfully applied to a  quantity of industrial processes, see  e.g. \cite{pr11,pr24,su85,va95}.

A RFS contains four components—rules, fuzzifier, inference, and output processor  as shown in Figure 1. Once the rules have been established, the fuzzy system can be viewed as a mapping from inputs to outputs.  When T1 (T2, resp.) FSs are used to model the linguistic variable that appears in rule's antecedent and consequent, the RFS is  called  a T1 (T2, resp.) RFS.   In recent decades, many authors have studied T2 RFSs \cite{ch21,ku14,me24,su25}, especially interval-valued  (IV)  RFSs \cite{co06, fe16} (since each IV FS  can be expressed by a T2 FS,    IV RFSs can be viewed as one kind of  T2 RFSs).  Note that  in the research   of  T2 RFSs,  the membership  degrees of  T2 FSs  are usually assumed to be   convex normal functions from \([0,1]\) to \([0,1]\).  

Triangular norms on the unit interval \([0,1]\) were introduced  by Schweizer and Sklar \cite{sc83}. Due to the close connection between order theory and fuzzy set theory, see e.g. \cite{go67}, t-norms on generally  partially ordered sets (posets) have been studied by many authors, see e.g. \cite{la23, de04, zh05}.

\begin{figure}[!t]
    \centering
    \includegraphics[width=0.6\textwidth]{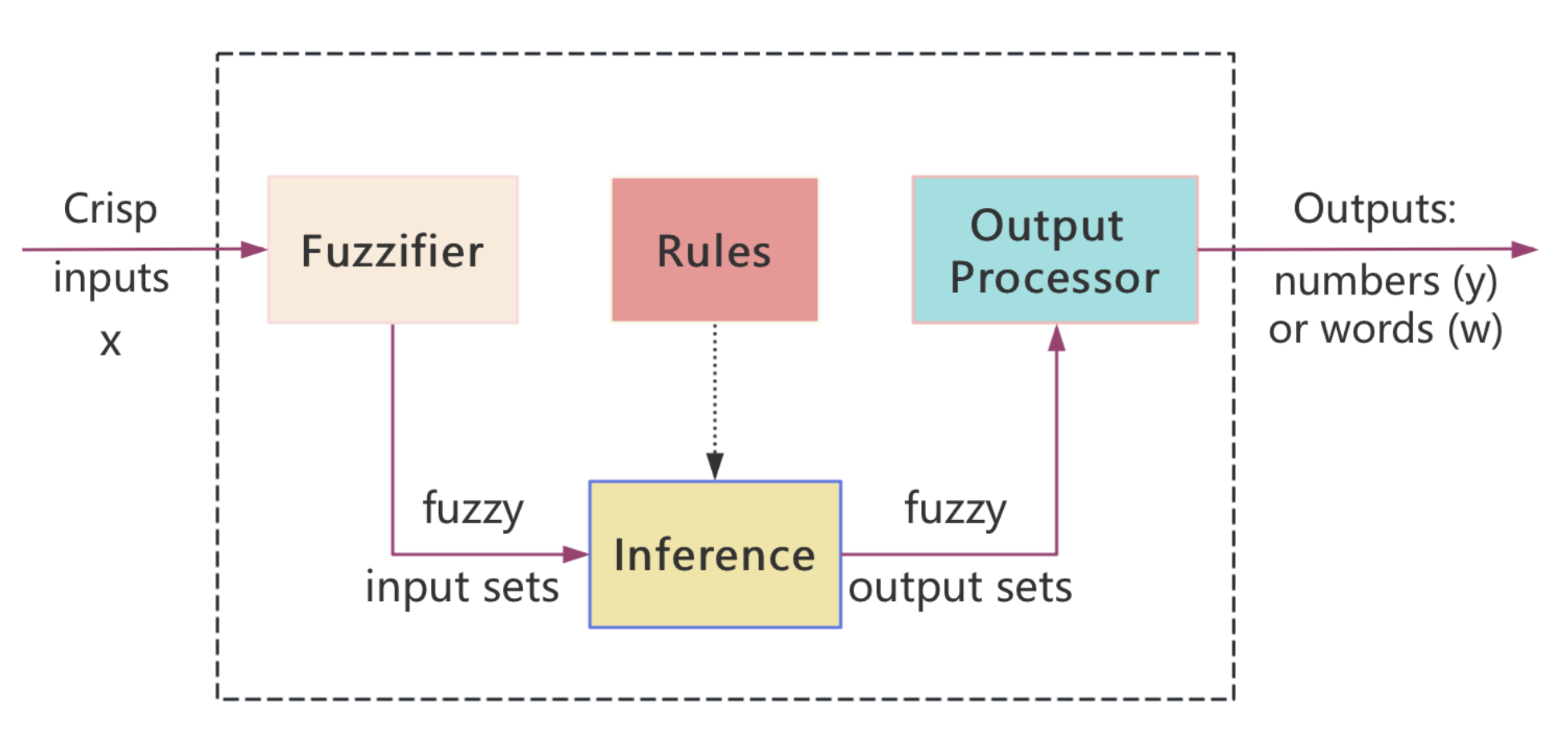}
    \caption{Rule-based fuzzy system}
    \label{fig}
\end{figure}

Triangular norms on complete lattices  play an important role in the inference block of  RFS.  For example, suppose that there is only one rule in Figure 1, that is  \[ \text{IF input is \(A\), THEN output is \(B\)}.\]   According to Mamdani implication (see \cite{me24,ma75})  \[\mu_{A\rightarrow B}(x,y)=A(x)\ast B(y)\]  and Zadeh's compositional rule of inference (see \cite{fu93,kl00,za73}), the inference block assigns  a  fuzzy  input set \(A'\)  to a  fuzzy  output set  \(B'\) given by  \[B'(y)= \sup\{A'(x) \ast A(x) \ast B(y)\mid x\in X\},\] where \(\ast\) is a t-norm on certain complete lattice.  Concretely,
\begin{itemize}
\item In  T1 RFS,    all fuzzy sets \(A, B, A', B'\) are type-1 fuzzy sets,  and \(\ast\) is a t-norm on complete lattice \(([0,1],\le)\) such as minimum or product. 
\item In  IV RFS,    all fuzzy sets \(A, B, A', B'\) are  interval-valued fuzzy sets,  and \(\ast\) is a t-norm on complete lattice  \((\mathbf{I}^{[2]},\le_I)\) (\(\mathbf{I}^{[2]}\) is the set of all closed subintervals of \([0,1]\),  and \(\le_I\) is the interval-valued order \cite{de03}).
\item  In T2 RFS,   all fuzzy sets \(A, B, A', B'\)  are type-2 fuzzy sets whose membership degrees are elements of  \(\mathbf{L}\),  and \(\ast\) is a t-norm on complete lattice \((\mathbf{L}, \sq)\). 
\end {itemize}
Moreover, there exists  order-preserving embedding from lattice \(([0,1],\le)\) to lattice  \((\mathbf{I}^{[2]},\le_I)\), from lattice \((\mathbf{I}^{[2]},\le_I)\)  to lattice \((\mathbf{L}, \sq)\), respectively, which  means that IV RFS is the generalization of   T1 RFS, and  T2 RFS is the generalization of   IV RFS to some extent. 

As mentioned above,  t-norms  play an important role in RFS  and may influence  performance of  RFS.  To provide more choice for  design of  RFS, it is  necessary to study  construction of  t-norms.  However, though 
 the construction of t-norms on lattice \(([0,1],\le)\) and   \((\mathbf{I}^{[2]},\le_I)\)   have been deeply studied for T1  RFSs and IV RFSs, respectively,  the same has not happened with  t-norms on \((\mathbf{L},\sq)\) for T2 RFSs. 
 
To construct operators on \(\mathbf{L}\) or \(\mathbf{M}=[0,1]^{[0,1]}\) (all maps from [0,1] to [0,1]),  
the  mainstream method  is based on  convolution (see Definition \ref{df}), which has a close relation with  Zadeh's extension principle \cite{za75}. Convolution were widely studied as   (generalized) extended t-norm  \cite{he22,wu20,zh19}, uninorm \cite{li22,xi18}, implication \cite{wa15,wa18}, aggregation\cite{ta14,to17}, overlap function \cite{ji22,li20}, etc.
 In particular, through the research on convolutions \(\wedge_\wedge\) (extended minimum) and \(\vee_\wedge\) (extended maximum),  a series of significant properties of  poset \((\mathbf{M},\sq)\) and its subposets, such as  (\(\mathbf{L},\sq\)), have been obtained by Harding, Walker and Walker (see \cite{ha08,ha10,wa05}).  Besides, Zhang and Hu \cite{zh25} showed that one kind of convolution \(\ast_\va\) can be calculated through  \(\alpha\)-cut.  Hu and Wang \cite{hu142} provided an equivalent characterization of  convolution order \(\sq\) on \(\mathbf{L}\) through the strong \(\alpha\)-cut, that makes the convolution order \(\sq\) more intuitive.

To construct t-norms on \((\mathbf{L},\sq)\) for T2 RFS,   a key problem appears naturelly, when   convolution is   a t-norm on \((\mathbf{L},\sq)\). In 2014,  Hernández,   Cubillo  and Torres-Blanc \cite{he14} provided  sufficient but not necessary  conditions under which  convolution  is a \(t_r\)-norm (t-norm satisfying more restrictive axioms) on \((\mathbf{L},\sq)\).  In this paper, we  completely solve the above problem.  
The main results are listed as follows.
\begin{itemize}
 \item[1)]  The  necessary and  sufficient  conditions, under which convolution is a t-norm on \((\mathbf{L},\sq)\),  is obtained, see Theorem \ref{24} and Theorem \ref{22}.  A large number of  t-norms on \((\mathbf{L},\sq)\) can be constructed  by  Theorem \ref{22}. As an application, experts can choice  t-norms more flexibly  to design  T2 RFS. 
  
  \item[2)] One kind of  convolutions are obtained which are t-norms on \((\mathbf{L},\sq)\) and extremely easy to calculate, see Definition \ref{d6} and Proposition \ref{p}. As a result, T2 RFSs, employing  this kind of convolutions, are relatively easy to compute.
     
 \item[3)] New equivalent characterization of convolution order \(\sq\) are given, see   Proposition \ref{17}. This result may  promote the research of  aggregation  operators (one kind of monotone functions)  on \((\mathbf{L},\sq)\).   Aggregation refers to the process of combining several values into a single one, and are used for the fusion of information \cite{gr09}. In particular, arithmetic mean  is an aggregation  operator.
\end{itemize}

 The paper are organized as follows. Section II reviews  definitions and properties of t-norm, \(t_r\)-norm, type-2 fuzzy set and convolution. Section III shows that \(\ast_\va\) is a t-norm if and only if  it is a \(t_r\)-norm, and provides the necessary conditions under which convolution  \(\ast_\va\) is a t-norm. Section IV  proves that the necessary conditions given in  Section III  is also sufficient.  Section V provides the necessary and sufficient conditions under which   \(\ast_\va\) is a t-conorm or \(t_r\)-conorm.

\section{Preliminaries}
In this section, we recall some  definitions and properties which are used in the sequel.
 \subsection{T-norm and L-fuzzy set}
Let \(\ast\) be a binary operation on  \(X\). The algebra \((X,\ast)\) is called a \textit{semigroup}, if for all \(x,y,z\in X\), \((x\ast y)\ast z= x\ast(y\ast z)\). 
The algebra \((X,\ast)\) is called \emph{commutative}, if  for all \(x,y\in X\), \(x\ast y=y\ast x\).
An element \(e\in X\) is called the \emph{unit} of the algebra  \((X,\ast)\), if \(e\ast x=x\ast e =x\) for all \(x\in X\).

Let \(\le\) be a   partial order  on \(X\).  The triple \((X,\ast,\le)\) is called a \emph{partially ordered semigroup} or \emph{posemigoup} for brief,  if  \((P,\ast)\) is a semigroup such that  \(x\ast y\le x\ast z,\) \(y\ast x\le z\ast x\) whenever \(y\le z\).  

\begin{defn}
Let  \((P,\le,0_P,1_P)\) be a bounded poset. A binary operation \(\ast:P^2\to P\) is called  a t-norm (t-conorm resp.),  if  \((P,\ast,\le)\) is a  commutative posemigroup with \(1_P\) (\(0_P\) resp. ) being the unit.
\end{defn}

Two posemigroups \((P_1,\ast_1, \le_1)\) and \((P_2,\ast_2, \le_2)\) are called to be \emph{isomorphic}, if there is an order isomorphism \(\tau: P_1\to P_2\)  such that for all \(x,y\in P_1,\) \[\tau(x\ast_1 y)=\tau(x)\ast_2 \tau(y).\]
In particular, two t-norms \(\ast_1:P_1^2\to P_1\) and \(\ast_2:P_2^2\to P_2\) are called to be isomorphic, if the two posemigroups \((P_1,\ast_1, \le_1)\) and \((P_2,\ast_2, \le_2)\) are isomorphic.

\begin{prop}\label{2} 
Suppose that  posemigroups \((P_1,\ast_1,\le_1)\) and  \((P_2,\ast_2,\le_2)\)  are isomorphic. Then  \(\ast_1\) is a t-norm on \((P_1,\le_1)\) if and only if \(\ast_2\) is a t-norm on \((P_2,\le_2)\).
\end{prop}

\begin{prop}\label{..2} 
 Let \(\ast_1\) and \(\ast_2\) be two binary functions on bounded posets \((P_1,\le_1)\) and \((P_2,\le_2)\), respectively.   If  there exists an order-reversing isomorphism \( \tau : P_1\to P_2\) satisfying for all \(x,y\in P_1\), \[\tau (x\ast_1 y)=\tau(x)\ast_2 \tau(y),\] then \(\ast_1\) is  a t-norm on \((P_1,\le_1)\)   if and only if \(\ast_2\) is a t-conorm on  \((P_2,\le_2)\). 
\end{prop}

\begin{exmp}
The product t-norm \(T_p:[0,1]^2\to[0,1]\)  and  the {\L}uckasiewicz  t-norm \(T_{L}:[0,1]^2\to[0,1]\),  are respectively  given by \[T_p(x,y)=x \cdot y,\quad  T_{L}(x,y)=\max\{x+y-1,0\}.\]  
\end{exmp}

Suppose that \(0\le a<b\le 1\). A t-norm \(\ast:[a,b]^2\to [a,b]\) is called
\begin{itemize}
\item\emph{continuous}, if it is continuous as a two binary function.
\item \emph{border continuous}, if it is continuous on the boundary of the unit square \([a,b]^2\),  i.e. on the set of \([a,b]^2\setminus (a,b)^2\). 
\item \emph{left-continuous}, if it is left continuous in each component.  
\item  \emph{conditionally cancellative},  if  \(x_1\ast y= x_2\ast y> a\) implies  \(x_1=x_2\).  
\item \emph{Archimedean},  if for any \(x,y\in (a,b)\), there exists \(n\in \mathbb{N}\) such that \(x^{(n)}_\ast \le y\). 
\end{itemize}

Let \(\ast:[0,1]^2\to [0,1]\) be a binary map. For \(a,b\in[0,1]\),  define \[ a\ast b^-=\sup \{a\ast t\mid t< b\}\]  In particular, \(a\ast 0^-=\sup \varnothing=0.\)

 The following proposition is obtained from the monograph \cite{al06}.
\begin{prop}\label{p3}
Let \(\ast:[0,1]^2\to [0,1]\) be a t-norm. Then 

(i) \(\ast\) is continuous  if and only if  it is continuous in each component.

(ii) \(\ast\) is  border continuous if and only if  for all \(a\in [0,1]\), \(a\ast 1^-= a\ast 1\).

(iii)  \(\ast\) is left-continous implies that \(\ast\)  is  border continuous.
\end{prop}

Let  \(L\) be a complete lattice. An \emph{\(L\)-fuzzy set} \(A\) is determined  by a  membership function \[\mu_A:  X\to L\]  where \(\mu_A(x)\) denote the membership degree of \(x\in X\) to \(A\) (see \cite{go67}). 

In particular, type-1 fuzzy sets are exactly  \([0,1]\)-fuzzy sets.  Interval-valued fuzzy sets are exactly  \(\mathbf{I}^{[2]}\)-fuzzy sets, where  the complete lattice  \(\mathbf{I}^{[2]}\) denotes the set of all closed subintervals of \([0,1]\) equipped with the order \(\le_I\) defined by \([a_1,b_1]\le_I  [a_2, b_2]\), if  \(a_1\le b_1\) and \(a_2 \le b_2\).

\subsection{Type-2 fuzzy set and convolution }
A \emph{type-2 fuzzy set} \(A\) is determined  by a  membership function \[\mu_A:  X\to \mathbf{M}=[0,1]^{[0,1]}=\text{Map}([0,1],[0,1]).\]  

 A standard method to construct  operations on \(\mathbf{M}\)  is the convolution.
\begin{defn}\label{de} \label{df} Let   \(\ast:[0,1]^n\to[0,1]\) and \(\va:[0,1]^n\to[0,1]\) be two \(n\)-ary  operations on \([0,1]\). The \emph{convolution} induced by  \(\ast\) and  \(\va\) is an \(n\)-ary operation \(\ast_\va:\mathbf{M}^n\to \mathbf{M}\)  given by \[\ast_\va(f_1,\cdots,f_n)(x)=\bigvee_{\ast(y_1, \cdots, y_n)=x} \va(f(y_1),\cdots, f(y_n)).\label{1.1}\] In particular,  if there not exist \((y_1, y_2, \cdots, y_n)\in [0,1]^n\)
 such that \(\ast(y_1, \cdots, y_n)=x\),  let \[\ast_\va(f_1,\cdots,f_n)(x)=   \bigvee \varnothing= 0.\] \end{defn}

Let \(\ast\) be an \(n\)-ary operation  on \([0,1]\). The Zadeh extension of \(\ast\)  is exactly the convolution \(\ast_\wedge:\mathbf{M}^n\to \mathbf{M}\). 
The \(\va\)-extension of  \(\ast\) is exactly the  convolution \(\ast_\va:\mathbf{M}^n\to \mathbf{M}\), where \(\va\) is usually a t-norm on \([0,1]\). The \(\va\)-extension is also called  generalized  extension.

 For each  \(\alpha\in [0,1]\), the \(\alpha\)-cut and  strong \(\alpha\)-cut of  \(f\in \mathbf{M}\) are respectively given by  
\[f^{\alpha}=\{x\in[0,1]\mid f(x)\ge \alpha\},\]
\[f^{\hat{\alpha}}=\{x\in[0,1]\mid f(x)> \alpha\}.\]

If \(\ast\) be a binary function on \([0,1]\),  for any \(A,B\subseteq  [0,1]\),  let \(A \ast B\) denote the  set \[\{x\ast y\mid x\in A, y\in B\}.\]

Under the case that  \(\ast:[0,1]^2\to [0,1]\) is  surjective and  and \(\va:[0,1]^2\to [0,1]\) is a t-norm,   the convolution \(\ast_\va\) can be represented through  \(\alpha\)-cut, see \cite[Theorem 3.1 (2)]{zh25}. In the following, we give a more general result.

\begin{prop}\label{.14} 
If \(\ast\)  and \(\va\) are two binary functions on \([0,1]\) and  \(\va\) is  monotone  increasing in each place, then for any \( f,g\in \mathbf{M}\) and \(\alpha\in [0,1]\), \[ (f\ast_\va g)^{\hat{\alpha}}=\bigcup_{\alpha_1\va \alpha_2>\alpha} f^{\alpha_1}\ast g^{\alpha_2}.\]
\end{prop}
\begin{proof}
Firstly, we show that  \( (f\ast_\va g)^{\hat{\alpha}}\subseteq \bigcup_{\alpha_1\va \alpha_2>\alpha} f^{\alpha_1}\ast g^{\alpha_2}.\)  For any \(x_0\in  (f\ast_\va g)^{\hat{\alpha}}\), since \((f\ast_\va g)(x_0)>\alpha\), i.e.   \[(f\ast_\va g)(x_0)= \bigvee_{y\ast z=x_0}f(y)\va g(z)>\alpha\]  then there exist \(y_0,z_0\in [0,1]\) such that 
\(y_0\ast z_0=x\) and  \(f(y_0)\va g(z_0)>\alpha.\)  Put \(\alpha_1=f(y_0)\) and  \(\alpha_2=g(z_0),\) then  \[x=y_0\ast z_0\in  f^{\alpha_1} \ast  g^{\alpha_2}\subseteq  \bigcup_{\alpha_1\va \alpha_2>\alpha} f^{\alpha_1}\ast g^{\alpha_2}.\]

Conversely,  for  any  \(x\in  \bigcup_{\alpha_1\va \alpha_2>\alpha} f^{\alpha_1}\ast g^{\alpha_2},\)  there exist \(\alpha_1, \alpha_2\in[0,1]\) such that  \(\alpha_1\va \alpha_2>\alpha\) and  \(x\in f^{\alpha_1}\ast g^{\alpha_2}.\)  Due to \(x\in f^{\alpha_1}\ast g^{\alpha_2}\),   there exist \(y_0,z_0\in [0,1]\) such that \(x=y_0\ast z_0\), \(f(y_0)\ge \alpha_1\) and \(g(z_0)\ge \alpha_2.\) Therefore, \[(f\ast_\va g)(x)=\bigvee_{y\ast z=x}f(y)\va g(z)\ge  f(y_0)\va g(z_0)\ge \alpha_1 \va \alpha_2>\alpha,\] thus,   \(x\in  (f\ast_\va g)^{\hat{\alpha}}.\)
\end{proof}

\begin{prop}\label{..5} For \(f,g\in \mathbf{M},\) the  following statements are equivalent.
\begin{itemize}
 \item[(i)]\(f=g.\)
 \item[(ii)]  \(f^{\alpha}=g^{\alpha}\) for all \(\alpha\in(0,1].\)
 \item[(iii)] \(f^{\hat{\alpha}}=g^{\hat{\alpha}}\) for all \(\alpha\in[0,1).\)
 \end{itemize}
\end{prop}

There exist two  important partial orders \(\sqsubseteq\) and \(\preceq\) on \(\mathbf{M}\)  induced by  the convolutions    \(\wedge_{\wedge}\) (extended minimum) and   \(\vee_{\wedge}\)(extended maximum), respectively. 

\begin{defn} The  partial orders \(\sq\) and \(\pp\) on \(\mathbf{M}\) defined  by \[f\sqsubseteq g, \text{ if  } f\wedge_{\wedge} g=f,\]
 \[f\preceq g, \text{ if  } f\vee_{\wedge} g=g,\]
are called  to be \emph{convolution orders}.
\end{defn}

 It is known that \((\mathbf{M}, \sqsubseteq)\) is an inf semilattice,  \((\mathbf{M},\pp)\) is  a sup semilattice, and   the convolution orders \(\sq\) and \(\pp\)  are distinct on \(\mathbf{M}\) (see \cite{wa05}).

For \(A\subseteq[0,1]\), let \(\bar{A}:[0,1]\to[0,1]\) denote the \emph{characteristic function} of \(A\), that is \[\bar{A}(x)=\begin{cases}1,\quad  \text{if } x\in A,\\ 0, \quad \text{otherwise}.\end{cases}\] In particular,  for \(a\in[0,1]\), let \(\bar{a}\) denote the characteristic function of the  singleton \(\{a\}\).

\begin{prop}\label{77}\cite{wa05} For any \(f\in \mathbf{M}\),  \(\bar{0}\preceq f\) and \(f \sqsubseteq \bar{1}\).
\end{prop}

Unfortunately, \((\mathbf{M},\sq,\bar{0}, \bar{1})\) is not a bounded poset  since \(\bar{0}\) is not the minimum, and \((\mathbf{M},\pp,\bar{0}, \bar{1})\) is not a bounded poset since \(\bar{1}\) is not the maximum.  

Note that some significant results of the  subposets of  \((\mathbf{M},\sq)\) (or \((\mathbf{M},\pp)\)) are obtained by Harding, Walker and Walker (see \cite{ha08,ha10,wa05}). In the following, we recall these  subposets.

An element \(f\) of  \(\mathbf{M}\) is called \emph{normal}, if \[\sup \{f(x)\mid x\in [0,1]\}=1.\]
Let \(\mathbf{N}\) denote the set of all normal elements of \(\mathbf{M}\).

\begin{prop}\label{88}\cite{he14} For each \(f\in \mathbf{M}\),   \(\bar{0}\sqsubseteq  f \Leftrightarrow f \preceq\bar{1}\Leftrightarrow f\in \mathbf{N}\).
\end{prop}

The conditions, under which the subset of \(\mathbf{M}\) is a bounded poset, can be  obtained from Proposition \ref{77} and Proposition \ref{88} immediately.

\begin{prop}
Let \(\{\bar{0},\bar{1}\} \subseteq P \subseteq \mathbf{M}\).  Then \((P,\sq,\bar{0},\bar{1})\) (or \((P,\pp,\bar{0},\bar{1})\))  is a bounded poset  if and only if \(P\subseteq \mathbf{N}\).
\end{prop}

An element \(f\) of  \(\mathbf{M}\) is called \emph{convex}, if  \[x\le y\le z \Rightarrow  f(x)\wedge f(z) \le f(y).\]  Let \(\mathbf{C}\) denote the set of all convex elements of \(\mathbf{M}.\)   

For \(f\in \mathbf{M}\), let \(f^L,f^R\in \mathbf{M}\) be given by
 \[f^L(x)=\sup\{f(t)\mid  t\le x\},\quad f^R(x)=\sup\{f(t)\mid t\ge x\}.\]

 There exist some equivalent statements of the concept  \textit{convex}.  (see e.g. \cite{wa05}).
 
 \begin{prop}\label{10}For \(f\in \mathbf{M}\), the following statements are equivalent.
\begin{itemize}
\item[(i)] \(f\) is convex.
\item[(ii)] \(f=f^L\wedge f^R\).
\item[(iii)]  \(f^\alpha\) is a convex set for each \(\alpha\in (0,1]\).
\item[(iv)]  \(f^{\hat{\alpha}}\) is a convex set for each \(\alpha\in [0,1)\).
 \item[(v)] \(f\) is the minimum of a monotone increasing function and a monotone decreasing one.
\end{itemize}
\end{prop}

 Let \(\mathbf{L}\) denote the set of all normal  convex elements of \(\mathbf{M}\).  Then the convolution orders \(\sqsubseteq\)  and \(\preceq\) coincide on \(\mathbf{L}\) (see \cite[Proposition 19, Proposition 37]{wa05}), and the bounded poset \((\mathbf{L},\sqsubseteq,\bar{0},\bar{1})\) is a completely distributive lattice (see \cite{ha08}). 
 
 \begin{prop}\label{5}\cite{ha08}
 For any \(f,g\in \mathbf{L}\), \(f\sq g\) if and only if \( g^L\le f^L\) and \( f^R\le g^R\).
 \end{prop}

An element \(f\) of  \(\mathbf{M}\) is called \emph{upper semicontinuous}, if  \(f^{\alpha}\) is closed for any \(\alpha\in [0,1]\). Let \(\mathbf{L}_\mathbf{u}\) denote the set of all normal, convex  and upper  semicontinuous elements of \(\mathbf{M}\).   The bounded poset \((\mathbf{L}_\mathbf{u},\sq,\bar{0},\bar{1})\) is a  completely distributive lattice (see \cite{ha10}).

Let \(\mathbf{J}\) denote the set of all characteristic functions of singletons in \([0,1]\), i.e.
\[\mathbf{J}=\{\bar{x}\mid x\in[0,1]\}.\]   There exists an order isomorphism between \(([0,1],\le)\) and \((\mathbf{J},\sq)\), that is \[\tau: [0,1]\to \mathbf{J},\quad x\mapsto \bar{x}.\]

Let \(\mathbf{J}^{[2]}\) denote the set of all  characteristic functions of all closed subintervals of \([0,1]\), i.e. \[\mathbf{J}^{[2]}=\{\overline{[a,b]}\mid 0\le a \le b \le 1\}.\]
There also exists an order isomorphism between  \((\mathbf{I}^{[2]},\le_I)\) and \((\mathbf{J}^{[2]},\sq)\), that is \[\tau:\mathbf{I}^{[2]}\to \mathbf{J}^{[2]},\quad [a,b]\mapsto \overline{[a,b]}.\]

Obviously, \[\mathbf{J}\subseteq \mathbf{J}^{[2]}\subseteq \mathbf{L}_\mathbf{u}\subseteq \mathbf{L}=\mathbf{C}\cap \mathbf{N}\subseteq \mathbf{M}.\]
\begin{defn}\cite{he14}
A binary operation \(T : \mathbf{L}^2\to \mathbf{L}\) is called a
\(t_r\)-norm on \((\mathbf{L},\sq)\),  if \(T\) is a t-norm on \((\mathbf{L},\sq)\), \(T\) is closed on \(\mathbf{J}\) and \(\mathbf{J}^{[2]},\) respectively, and \(T(\overline{[0, 1]}, \overline{[a,b]})=\overline{[0,b]}.\)

 A binary operation \(S : \mathbf{L}^2\to \mathbf{L}\) is called a
\(t_r\)-conorm on \((\mathbf{L},\sq)\), if \(S\) is a t-conorm on  \((\mathbf{L},\sq)\), \(S\) is closed on \(\mathbf{J}\) and \(\mathbf{J}^{[2]},\) respectively, and \(S(\overline{[0, 1]}, \overline{[a,b]})=\overline{[a,1]}.\)
\end{defn}

 There exist two reasons to concern about  the poset \((\mathbf{L},\sq)\). On the one hand,  the  convolution orders \(\sqsubseteq\)  and \(\preceq\) coincide on \(\mathbf{L}\), and the bounded poset \((\mathbf{L},\sqsubseteq,\bar{0},\bar{1})\) is a complete  lattice. On the other hand, \(\mathbf{J}\subseteq \mathbf{J}^{[2]}\subseteq \mathbf{L}\) means that  \(\mathbf{L}\)-fuzzy sets are generalization of  type-1 fuzzy sets and interval-valued fuzzy sets.

\emph{ In the following context,  symbols  \(\ast\) and \(\va\) always denote any two  binary operations on \([0,1]\), and \(\ast_\va: \mathbf{M}^2\to \mathbf{M}\) denotes the convolution induced by \(\ast\) and \(\va\), i.e.}
\begin{equation}(f\ast_\va g)(x)=\bigvee_{y\ast z=x} f(y)\va g(z).\label{1.1}\end{equation}

\section{Necessity}

In this section, we  present the necessary conditions such that \(\ast_\va\) is a t-norm on \((\mathbf{L},\sq)\). Besides, we prove that convolution \(\ast_\va\) is a t-norm on \((\mathbf{L},\sq)\)  if and only if  it is a \(t_r\)-norm on \((\mathbf{L},\sq)\).

In order to prove Proposition \ref{8}, we  prove some relevant lemmas as follows.

\begin{lem}\label{p9}
If for all \(x\in [0,1]\), \(\bar{1}\ast_\va \bar{x} =\bar{x}\), then  for any \(A,B\subseteq [0,1]\), \[\bar{A}\ast_\va \bar{B}=\overline{A\ast B}.\]
\end{lem}
\begin{proof}
Firstly, we show that  for any \(a,b\in \{0,1\}\), \(a\va b=a\wedge b\).   If  \(1\va 0>0\),  then for any two distinct \(x, y\in[0,1)\),   by (\ref{1.1})
 \[(\bar{1} \ast_\va \bar{x})(1\ast y)\ge \bar{1}(1)\va \bar{x}(y) =1\va 0>0,\]  thus, \[1\ast y\in (\bar{1} \ast_\va \bar{x})^{\hat{0}}=\bar{x}^{\hat{0}}=\{x\},\] so \(1\ast y=x\).  However, it is impossible that for any two distinct \(x, y\in[0,1)\),   \(1\ast y=x\).  Therefore,  \(1\va 0=0.\)
Similarly, we can get  \[0\va 1=0\va 0=0.\]  Due to
\begin{align*}
1&=(\bar{1}\ast_\va \bar{x})(x)\\&
 =\bigvee_{y\ast z=x}\bar{1}(y)\va \bar{x}(z)\\&
 \le \sup(\Ran(\bar{1})\va \Ran(\bar{x}))\\&
 =  \max (\{0,1\}\va \{0,1\})\\&
 =\max \{0\va 0, 0\va 1, 1\va 0, 1\va 1\}\\&
 =1 \va 1, 
\end{align*} we have  \(1\va 1=1\).

Since for any \(a,b\in \{0,1\}\), \(a\va b=a\wedge b\),  it is easy to check that for any \(A,B\subseteq [0,1]\), \(\bar{A}\ast_\va \bar{B}=\overline{A\ast B}\) by (\ref{1.1}).
\end{proof}

\begin{lem}\label{2.15}
Let \(\mathbf{J}^{[2]}\subseteq P\subseteq \mathbf{L}.\) If \(\ast_\va\) is a  t-norm on \((P,\sq)\), then \(\ast_\va\) is closed on \(\mathbf{J}\) and \(\mathbf{J}^{[2]}\), respectively, and \(\ast\) is a continuous t-norm. 
\end{lem}
\begin{proof}
Firstly,  \(\ast_\va\) is closed on \(\mathbf{J}\). In fact, since \(\ast_\va\) is  a t-norm,  \(\bar{1}\ast_\va \bar{x} =\bar{x}\) holds for all \(x\in [0,1]\). By Lemma \ref{p9},  for all \(x,y\in [0,1]\), \begin{equation}\label{q2}\bar{x}\ast_\va\bar{y}=\overline{x\ast y}.\end{equation}  thus  \(\ast_\va\) is closed on \(\mathbf{J}\).  Trivially, \(\ast_\va\) is a t-norm on \((\mathbf{J},\sq).\)

Next, \(\ast_\va\) is closed on \(\mathbf{J}^{[2]}\). Given any \(A, B\in \mathbf{I}^{[2]}\),   by Lemma \ref{p9},  \[\overline{A\ast  B}=\bar{A}\ast_\va \bar{B}\in P\subseteq \mathbf{L}\]  which means 
 \[A\ast  B\in \mathbf{I}^{[2]},\]  i.e. \[\overline{A\ast B}\in \mathbf{J}^{[2]}.\]   Hence, \(\ast_\va\) is closed on \(\mathbf{J}^{[2]}\).
 
Finally, \(\ast\) is continuous t-norm.  It follows from  (\(\ref{q2}\)) that 
\[\tau: [0,1]\to \mathbf{J},  \quad x\mapsto \bar{x}\] is an isomorphism between   posemigroups \(([0,1],\ast,\le )\) and \((\mathbf{J}, \ast_\va, \sq)\). This, together with  that \(\ast_\va\) is a t-norm on  \((\mathbf{J},\sq)\),  implies that   \(\ast\) is a t-norm on \(([0,1],\le)\) by Proposition \ref{2}. In the above, we have shown that for any \(A,B\in \mathbf{I}^{[2]}\), \[A\ast  B\in  \mathbf{I}^{[2]}.\]  
Hence, for each \(a\in [0,1]\),  \[\{a\} \ast [0,1]\in \mathbf{I}^{[2]},\] which means that the  section \[a\ast (-) : [0,1]\to[0,1] \] is  continuous. Therefore, t-norm  \(\ast\) is  continuous.
\end{proof}

\begin{lem}\label{2.3} Let \(\mathbf{L}_\mathbf{u}\subseteq P\subseteq \mathbf{L}\). 
 If \(\ast_\va\) is a t-norm on \((P,\sq)\), then  \(\va\) is a t-norm.
\end{lem}
\begin{proof} 
For each \(a\in[0,1]\), define  \(p_a\in \mathbf{L}_\mathbf{u}\) as  follows, \[ p_a(x)=\begin{cases} 1, \quad  x=0,\\ a,\quad \text{otherwise.}
\end{cases}\] 
Define the set \(\mathcal{P}=\{p_a\mid a\in[0,1]\}\)  and the order-preserving isomorphism  \(p:([0,1],\le)\to (\mathcal{P},\sq)\), \[p(a)=p_a.\]
By Lemma \ref{2.15}, \(\ast\) is a t-norm.  Then by (\ref{1.1}), it is not difficult to check that for all \(a,b\in[0,1]\),  \[p(a)\ast_\va p(b)=p(a\va b),\]
thus, t-norm \(\ast_\va\) is closed on \(\mathcal{P}\), and  \((\mathcal{P},\ast_\va,\sq)\) is a commutative posemigroup.  Since the map \(p\) is an isomorphism between  commutative posemigroups \(([0,1],\va,\le) \)  and \((P,\ast_\va, \sq)\), and for all \(a\in[0,1],\)   \[a=p_a(1)=(p_a\ast_\va \bar{1})(1)=p_a(1)\va \bar{1}(1)= a\va 1,\]  then  \(\va\) is a t-norm.
\end{proof}

\begin{prop}\label{8} If \(\ast_\va\) is  a   t-norm on \((\mathbf{L}_\mathbf{u},\sq)\) or \((\mathbf{L},\sq)\),  then 
\begin{itemize}
\item[(i)] \(\ast\)  is a  continuous t-norm and \(\va\) is a t-norm.
\item[(ii)] \(\ast_\va\) is  closed on  \(\mathbf{J}\) and \(\mathbf{J}^{[2]}\), respectively.
\item[(iii)]  \(\overline{[0, 1]}\ast_\va \overline{[a,b]}=\overline{[0,b]}.\)
\end{itemize}
\end{prop}
\begin{proof} (i)  and (ii)  are straightforward from Lemma \ref{2.15} and Lemma \ref{2.3}. By the continuity of  \(\ast\) and  Lemma \ref{p9}, (iii) holds.
\end{proof}

 \begin{thm}  \label{24}
The  convolution \(\ast_\va\) is a  \(t_r\)-norm  on \((\mathbf{L},\sq)\) if and only if   it is    a \(t\)-norm on \((\mathbf{L},\sq).\)
\end{thm}
\begin{proof}
Obviously, each  \(t_r\)-norm is a t-norm. Another direction holds from Proposition \ref{8}.
\end{proof}

\begin{prop}\label{7}
If \(\ast_\va\) is  a  t-norm on \((\mathbf{L},\sq)\),  then \(\va\) is a border continuous t-norm.
\end{prop}
\begin{proof}
From Proposition \ref{8},  \(\ast\) is a continuous t-norm and \(\va\) is a  t-norm. 
For any \(a\in[0,1]\), let \(f,g,h\in \mathbf{L}\) be given by \(f\equiv 1,\) \[g(x)=\begin{cases} x,\quad  x\in[0,1),\\ 0, \quad x=1,\end{cases}\]    \[h(x)=\begin{cases} 1,\quad x=0,\\ a, \quad x\in(0,1].\end{cases}\]

 Claim 1,  for each  \(x_0\in(0,1)\),
\((f\ast_\va g)\ast_\va h)(x_0)\ge a\). In fact, for each \(y\in[x_0,1)\), since  \(\ast\) is continuous, there exists \(z\in[0,1]\) such that \[y\ast z=x_0,\] and thus, \[(f\ast_\va g)(x_0)=\bigvee_{s\ast t=x_0}f(x)\ast g(t)\ge f(z)\va g(y)=1\va g(y)=y.\]
Hence,  \[(f\ast_\va g)(x_0)=1.\]  Therefore, \[((f\ast_\va g)\ast_\va h)(x_0)\ge (f\ast_\va g)(x_0)\ast h(1)=1\ast a= a.\]

Claim 2,  for each \(x_0\in(0,1)\), \((f\ast_\va (g\ast_\va h))(x_0)\le a\va 1^-.\) For this purpose, we only need to check that  for any \(y_0,z_0\in  [0,1]\) satisfying \(y_0\ast z_0=x_0,\) \begin{equation}\label{..6}f(y_0)\ast (g\ast_\va h)(z_0)\le a\va 1^-.\end{equation}

In fact, for all \(z_1,z_2\in [0,1]\) satisfying \[z_1\ast z_2=z_0,\] since  \[z_1\ast z_2\ast y_0=z_0\ast y_0=x_0>0,\] then \( z_2>0\),  thus,  \[g(z_1)\va h(z_2)=g(z_1)\va  a \le a\va 1^-.\]

In the above,  we have  shown that  for all   \(z_1,z_2\in [0,1]\) satisfying \(z_1\ast z_2=z_0,\) \[g(z_1)\va h(z_2)\le a \va 1^-.\]  Hence, \[(g\ast_\va h)(z_0)\le a\va 1^-.\] Therefore, (\ref{..6}) holds.

By claim 1 and claim 2, for any \(x_0\in(0,1)\),   
\[a\le ((f\ast_\va g)\ast_\va h)(x_0)= (f\ast_\va (g\ast_\va h))(x_0)\le a\va 1^-.\]
On the other hand, since \(\va\) is a t-norm, \(a\ge a\va 1^-\). Therefore, \(a=a\va 1^-\), i.e. t-norm  \(\va\) is border continuous.
\end{proof}

\begin{prop}\cite{al06}  Suppose that \(0\le a<b\le 1\). A t-norm \(\ast:[a,b]^2\to [a,b]\)  is  Archimedean continuous   if and only if  it is isomorphic to product t-norm \(T_P\) or  {\L}uckasiewicz t-norm \(T_L\).
\end{prop}

\begin{prop}\label{0.20}\cite{al06} A binary function  \(T:[0,1]^2\to [0,1]\) is a continuous t-norm if and only if there exists a family of pairwise disjoint  open subintervals  \((a_\alpha, b_\alpha)_{\alpha\in A}\) of  \([0,1]\) and a family of  continuous   Archimedean t-norms  \(T_\alpha: [a_\alpha, b_\alpha]^2\to  [a_\alpha, b_\alpha]\)  such that \(T\) can be represented  by \[  T(x,y)=\begin{cases} T_\alpha(x,y), & (x,y)\in[a_\alpha, b_\alpha]^2,\\ x\wedge y, & \text{otherwise.}
\end{cases}
\]
\end{prop}
Obviously, if the index set \(A=\varnothing\) , then \(T=\wedge\) in the above proposition. 

\begin{lem}\label{0.5}
Let \(\ast\) be a continuous t-norm not equal to \(\wedge\).  Then there exist \(0\le \alpha<\beta\le 1\) such  that  the restriction of \(\ast\) on \([\alpha,\beta]\) is a continuous Archimedean t-norm. Besides, the following statements hold.

(i) If   \(a\ast b\in (\alpha, \beta)\), then \(a,b\in  [\alpha,\beta]\).

(ii)  if \(a_1\ast b_1=a_2\ast b_2\in   (\alpha,\beta)\), then \(a_1>a_2\) implies  \(b_1<b_2\).
\end{lem}

\begin{proof}
By Proposition \ref{0.20}, (i) holds immediately. 
Since the restriction of  \(\ast\) on \([\alpha,\beta]\) is isomorphic to \(T_p\) or \(T_L\), it is conditionally cancellative. Hence, (ii) holds.
\end{proof}

\begin{prop}\label{0.21} If \(\ast_\va\) is  a  t-norm on \((\mathbf{L},\sq)\) and \(\ast\ne \wedge\), then   \(\va\) is a left-continuous t-norm.
\end{prop}
\begin{proof}
By Proposition \ref{7}, \(\va\) is a border continuous t-norm.   Given any \(u, v\in (0,1)\), we shall show  \(u\va v= u\va v^-\). Hence, \(\va\) is a left-continuous t-norm.

By Proposition \ref{8},  \(\ast\) is a continuous t-norm. Since \(\ast\neq \wedge\),   there exist  \(0\le \alpha<\beta\le 1\) such that the restriction of \(\ast\) on \([\alpha,\beta]\) is a
continuous Archimedean  t-norm. Choose arbitrary \(\la \in (\alpha,\beta)\)  satisfying \begin{equation}\label{e3}\la_{\ast}^{(3)}\in (\alpha,\beta),\end{equation} where \(\la_{\ast}^{(3)}\) denotes \(\la\ast \la \ast \la.\)
  Let the functions \(f,g,h\in \mathbf{L}\) be given by \[ f(x)=\begin{cases} \max\{0,x-\la+v\},  &  x\in (\la\ast \la, \la),\\1, & x=\la,\\ 0, & \text{otherwise},\end{cases}\] \[g(x)=\begin{cases}
  -x+\la+1, & x\in(\la,1],\\0, &\text{otherwise,}  \end{cases}\] \[h(x)=\begin{cases} u,&x\in[\la,1),\\ 1,& x=1,\\0,& \text{otherwise.}\end{cases}\]

Claim 1, \(((f\ast_\va g)\ast_\va h)(\la_{\ast}^{(3)})\ge u\va v\). Since the restriction of \(\ast\) on \([\alpha,\beta]\) is  a t-norm isomorphic to \(T_p\) or \(T_L\), and \[\la\ast \la\in (\alpha,\beta),\] then there exists a  strictly increasing sequence \(\{x_n\}_{n\in \mathbb{N}} \)  and a strictly  decreasing  sequence \(\{y_n\}_{n\in \mathbb{N}}\) in \((\alpha,\beta)\), both of them converge to \(\la\), and for all \( n\in\mathbb{N}\)\begin{equation}\label{q10}x_n\ast y_n= \la\ast \la.\end{equation} 
It is easy to see that both  \(\{f(x_n)\}_{n\in \mathbb{N}}\) and \(\{g(x_n)\}_{n\in \mathbb{N}}\) are monotone increasing,  and  \[\lim_{n\rightarrow \infty} f(x_n)=f(\la^-)=v, \quad \lim_{n\rightarrow \infty} g(y_n)=g(\la^+)=1.\] Since \(\va\) is a border continuous t-norm, 
 \[\lim_{n\rightarrow \infty} f(x_n)\va g(y_n) = \lim_{n\rightarrow \infty} f(x_n) \va \lim_{n\rightarrow \infty} g(y_n)= v\va 1=v.\]
 This, together with  (\ref{q10}),  implies that
 \[(f\ast_\va g)(\la\ast \la)=\bigvee_{y\ast z=\la\ast \la} f(y)\va g(z)\ge \lim_{n\rightarrow \infty} f(x_n)\va g(y_n) =v.\] Hence, \[((f\ast_\va g)\ast_\va h)(\la_{\ast}^{(3)})\ge (f\ast_\va g)(\la\ast \la)\va h(\la)\ge v\va u.\]
 
Claim 2, \((f\ast_\va (g\ast_\va h))(\la_{\ast}^{(3)})\le u\va v^-\).    For this purpose,  we only need to  show that for all \(s,t\in [0,1]\)  satisfying \begin{equation}\label{e5} s\ast t=\la_{\ast}^{(3)},\end{equation}
the following inequality holds
  \begin{equation}\label{q3}f(s)\va (g\ast_\va h)(t)\le u\va v^-\end{equation}   case by case.

Case 1, \(s\notin (\la\ast \la, \la].\) Since \(f(s)=0,\) then   (\ref{q3}) holds.

Case 2,   \(s=\la.\)  Since  the restriction of \(\ast\) on \([\alpha,\beta]\) is  a conditionally  cancellative  t-norm, \(s\ast t=\la_{\ast}^{(3)},\) and \(\la_{\ast}^{(3)}\in(\alpha, \beta)\), then \[t=\la\ast \la.\] 
 Hence, to prove  (\ref{q3}),  we only need to show  \begin{equation}\label{q.10}(g\ast_\va h)(\la\ast \la)=0.\end{equation}
For all \(b,c\in  [0,1]\) satisfying \[  b\ast c =\la\ast \la,\]  if \(b\le \la\), then  \[g(b)\va h(c)=0\va h(c)=0,\]
if  \(b>\la,\)  since \(b\ast c =\la\ast \la\),  we have   \(c<\la\) from Lemma \ref{0.5}, and thus,  \[g(b)\va h(c)=g(b)\va 0=0.\]  In the above, we have shown that for all \(b,c\in  [0,1]\) satisfying \( b\ast c =\la\ast \la,\) \[g(b)\va h(c)=0.\]
Hence, (\ref{q.10}) holds.

Case 3,  \(s\in(\la\ast \la, \la)\). Since \[f(s)<f(\la^-)= v,\]  to prove (\ref{q3}), we only need to show  \begin{equation}(g\ast_\va h)(t)\le u.\label{q.11}\end{equation}   
Suppose that \(b,c\in  [0,1]\) and \[b\ast c=t.\]  If \(c<1\), then \(h(c)\le u\), hence  \[g(b)\va h(c)\le u.\]   If \(c=1\), then \(b=t\). In this case, on the one hand, \begin{equation}\label{e9}g(b)\va h(c)=g(t)\va h(1)=g(t)\va 1=g(t).\end{equation} On the other hand,  by (\ref{e3}) and (\ref{e5}),  \[s\ast t=\la_{\ast}^{(3)}\in (\alpha,\beta).\]  This, together with  \(s>\la\ast \la,\) implies  \(t<\la\) from Lemma \ref{0.5}. Due to \(t<\la\), \[g(t)=0.\] Therefore, by (\ref{e9}), \[g(b)\va h(c)=g(t)=0\le u.\]
In the above, we have shown that  for all \(b,c\in [0,1]\) with \(b\ast c=t\),    \[g(b)\va h(c)\le u.\] Hence,   (\ref{q.11}) holds.

By claim 1 and claim 2,  \[u\va v =((f\ast_\va g)\ast_\va h)(\la_{\ast}^{(3)})=(f\ast_\va (g\ast_\va h))(\la_{\ast}^{(3)})\le u\va v^-.\] Since \(\va\) is a t-norm, \(u\va v \ge u\va v^-.\) Hence, \(u\va v=u\va v^-.\)  Therefore,  \(\va\) is left-continuous.
\end{proof}

As a conclusion  of this section, we give the necessary conditions such that \(\ast_\va\) is  a t-norm  on \((\mathbf{L},\sq)\).

\begin{prop}\label{.21}Let \(\ast_\va\) be  a  t-norm on \((\mathbf{L},\sq)\).  Then 
 \(\ast\)  is a  continuous t-norm, and \(\va\) is a border continuous t-norm.
Besides,  \(\va\) is  left-continuous  whenever \(\ast\ne \wedge\). 
 \end{prop}
\begin{proof} It follows from  Proposition \ref{8}, Proposition \ref{7} and  Proposition \ref{0.21}.
\end{proof}

\section {Sufficiency}
In this section, we   show that the necessary conditions given in Proposition \ref{.21}  is  also sufficient such that \(\ast_\va\) is a t-norm on \((\mathbf{L},\sq)\). Hence, we obtain the  necessary and sufficient  conditions as desired.  

\subsection{Closure}

In this subsection, we  give the sufficient conditions such that \(\ast_\va\) is closed on \(\mathbf{L}\).

\begin{lem}\label{9}
Suppose that \(\{A_i\}_{i\in  \Lambda}\) is a family of  convex subsets of  \([0,1]\).  If   \(A_i\cap A_j\ne \varnothing\) whenever \(A_i\ne \varnothing\) and \(A_j\ne \varnothing\),  then   \(\bigcup_{i\in  \Lambda} A_i\) is convex.
\end{lem} 
\begin{proof}
For any \(i,j\in \Lambda, x\in A_{i},y\in A_{j}\), 
since \(A_{i}\cap A_{j}\ne \varnothing\), and both \(A_i\) and \(A_j\) are convex, then 
\(A_{i}\cup A_{j}\) is convex, hence   \[[x\wedge y, x\vee y]\subseteq A_{i}\cup A_{j}\subseteq \bigcup_{i\in  \Lambda} A_i.\] Therefore,   \(\bigcup_{i\in  \Lambda} A_i\) is convex.
\end{proof}

\begin{prop}\label{.19}
 If  t-norm \(\ast\) is continuous  and t-norm \(\va\) is  continuous at \((1,1)\), then  \(\ast_\va\) is closed on \(\mathbf{L}\).
\end{prop}
\begin{proof}
 For any \(f,g\in \mathbf{L},\)    \(f\ast_\va g\) is normal. In fact,
 for any \(a\in[0,1)\), since \(\va\) is continuous at \((1,1)\), then there exist  \(b_1,c_1\in[0,1)\) satisfying \(b_1\va c_1>a.\)  
 By Proposition \ref{.14},  \[(f\ast_\va g)^{\hat{a}}=\bigcup_{b\va c>a} f^{b}\ast g^{c}\supseteq f^{b_1}\ast g^{c_1}\ne \varnothing.\] Therefore, \(f\ast_\va g\) is normal.
 
 Next,  \(f\ast_\va g\) is convex.   
 Given any  \(a\in[0,1),\)  we claim that for all \(b_1,b_2,c_1,c_2\in[0,1],\)
 if \(b_1\va c_1>a\), \(b_2\va c_2>a,\)  and both \(f^{b_1}\ast g^{c_1}\) and \( f^{b_2}\ast g^{c_2}\) are non-empty,   then \((f^{b_1}\ast g^{c_1})\cap  (f^{b_2}\ast g^{c_2})\) is non-empty. In fact, on the one hand, \[f^{b_1}\ast g^{c_1}\supseteq (f^{b_1}\cap f^{b_2})\ast (g^{c_1}\cap g^{c_2})\]
 and
   \[f^{b_2}\ast g^{c_2}\supseteq (f^{b_1}\cap f^{b_2})\ast (g^{c_1}\cap g^{c_2})\]
imply that 
 \[( f^{b_1}\ast g^{c_1})\cap ( f^{b_2}\ast g^{c_2})\supseteq (f^{b_1}\cap f^{b_2})\ast (g^{c_1}\cap g^{c_2})= f^{b_1\vee b_2}\ast g^{c_1\vee c_2}.\]
 On the other hand, since  \(f^{b_1}\ast g^{c_1}\) and \( f^{b_2}\ast g^{c_2}\) are non-empty, all \(f^{b_i},  g^{c_i}, i=1,2\) are non-empty. Therefore,  \((f^{b_1}\ast g^{c_1})\cap  (f^{b_2}\ast g^{c_2})\) is non-empty as desired.

By Lemma \ref{9},  the above claim implies that   \[\bigcup_{b\va c>a} f^{b}\ast g^{c}=(f\ast_\va g)^{\hat{a}}\] is convex.
 By Proposition \ref{10}, \(f\ast_\va g\) is convex. Hence, \(f\ast_\va g\in\mathbf{L}\).
\end{proof}

\subsection{Monotone}
In this subsection, we  show the sufficient conditions such that \(\ast_\va:\mathbf{L}^2\to \mathbf{L}\) is monotone increasing in each component.

A non-empty  convex subset of  \([0,1]\) is called an \emph{subinterval} of \([0,1]\). Let \(\mathcal{I}\) denote the set of all subinterval of \([0,1]\).  Recall that \(\mathbf{I}^{[2]}\)  is the set of all closed subintervals  of \([0,1]\), hence \(\mathbf{I}^{[2]}\subsetneqq\mathcal{I}\).

Define the   relation \(\p\) on \(\mathcal{I}\)  as follows.  
   \[A\p B,\text { if } A\wedge B=A.\]  
 Recall that  \(A\wedge B=\{x\wedge y\mid x\in A, y\in B\}\).
   
 \begin{exmp}
 Let \(0\le a_1< b_1\le 1\) and \(0\le a_2< b_2\le 1\). If  \(a_1\le a_2, b_1\le b_2\), then \([a_1,b_1]\p [a_2, b_2]\) and  \((a_1,b_1)\p (a_2, b_2)\).
 \end{exmp}  

The following property  were provided by Hu and Wang \cite{hu142}. Here, we give a new proof.
\begin{prop}\label{16}
For all \(f,g\in \mathbf{L}\), \(f \sqsubseteq g\) if and only if \(f^{\hat{a}}\p g^{\hat{a}}\)  for all  \(a\in[0,1)\).
\end{prop}
\begin{proof}
For any \(a\in[0,1)\),   \[(f\wedge_{\wedge} g)^{\hat{a}}=\bigcup_{b\wedge c>a} f^{b}\wedge g^c=(\bigcup_{b>a} f^{b})\wedge (\bigcup_{c>a} g^{c})=f^{\hat{a}}\wedge g^{\hat{a}},\]
thus \(f \sqsubseteq g\)\\
\(\Leftrightarrow\) \(f\wedge_{\wedge} g=f\)\\
\(\Leftrightarrow\) \((f\wedge_{\wedge} g)^{\hat{a}}=f^{\hat{a}}\) for all \(a\in[0,1)\)\\
\(\Leftrightarrow\) \(f^{\hat{a}}\wedge g^{\hat{a}}=f^{\hat{a}}\) for all \(a\in[0,1)\)\\
\(\Leftrightarrow\)   \(f^{\hat{a}}\p g^{\hat{a}}\)  for all \(a\in[0,1)\).
\end{proof}

\begin{defn}
For  any non-empty subset  \(A\)  of  \([0,1]\),  define   \[\ua A=\bigcup_{x\in A} [x,1]\] and \[\da A=\bigcup_{x\in A} [0,x].\] \end{defn}

Obviously, for each \(A\in \mathcal{I}\),  \(A=\ua A\cap \da A\).  Briefly,  we write  \(\da a\)  and \( \ua a\)  instead of  \(\da \{a\}\) and \(\ua \{a\}\), respectively. 
 
\begin{lem} \label{14}
For any  \(A, B\in \mathcal{I}\),  \(A\wedge B\in \mathcal{I}.\) Besides,  \(A=B\) if and only if \(\ua A= \ua B\) and \(\da A=\da B\).
\end{lem}

\begin{lem}\label{11} For any \(A, B\in \mathcal{I}\),  \[\ua (A\wedge B)=\ua A\cup \ua B, \quad \da (A\wedge B)=\da A\cap \da B.\]
\end{lem}
\begin{proof}
\(x\in \ua (A\wedge B).\) \\
\(\Leftrightarrow\) there exist \(a\in A, b\in B\) such that \(x\ge a\wedge b.\)\\
\(\Leftrightarrow\) there exist \(a\in A, b\in B\) such that \(x\ge a\) or \(x\ge b.\)\\
\(\Leftrightarrow\) \(x\in \ua A\) or  \(x\in \ua B\).\\
\(\Leftrightarrow\) \(x\in \ua A\cup \ua B\).

Similarly,  \(\da (A\wedge B)=\da A\cap \da B\).
\end{proof}

\begin{lem}\label{12} For \(A, B\in \mathcal{I}\), \(A\p  B\) 
if and only if \( \ua B \subseteq \ua A\) and \( \da A \subseteq \da B.\) 
\end{lem}
\begin{proof}

  \(A\p B\).\\
\(\Leftrightarrow\) \(A\wedge B=A\).\\
\(\Leftrightarrow\)  \(\ua (A\wedge B)= \ua A\) and  \(\da (A\wedge B)= \da A\) (by Lemma \ref{14}).  \\
\(\Leftrightarrow\)  \(\ua A\cup \ua B=\ua A\) and \(\da A\cap \da B=\da A\) (by Lemma \ref{11}).\\
\(\Leftrightarrow\) \( \ua B \subseteq \ua A\) and \( \da A \subseteq \da B.\) 
\end{proof}

\begin{lem} \label{13} Suppose that  \(\{A_i\}_{i\in \Lambda}\) is  a family of  subinterval of \( [0,1]\).  Then  \[\ua (\bigcup_{i\in \Lambda} A_i)=\bigcup_{i\in \Lambda} \ua A_i,\quad \da (\bigcup_{i\in \Lambda} A_i)=\bigcup_{i\in \Lambda} \da A_i .\]  Furthermore, if   \(\bigcap_{i\in \Lambda} A_i\ne \varnothing,\) then
 \[\ua (\bigcap_{i\in \Lambda} A_i)=\bigcap_{i\in \Lambda} \ua A_i,\quad \da (\bigcap_{i\in \Lambda} A_i)=\bigcap_{i\in \Lambda} \da A_i \] 
\end{lem}
\begin{proof}
 It is trivial that \[\ua (\bigcup_{i\in \Lambda} A_i)=\bigcup_{i\in \Lambda} \ua A_i,\quad \da (\bigcup_{i\in \Lambda} A_i)=\bigcup_{i\in \Lambda} \da A_i .\]  
In the following,  we shall  show that  if \(\bigcap_{i\in \Lambda} A_i\ne \varnothing,\)  then  \[\ua (\bigcap_{i\in \Lambda} A_i)=\bigcap_{i\in \Lambda} \ua A_i.\]
 Obviously, \[\ua  (\bigcap_{i\in \Lambda} A_i)\subseteq \bigcap_{i\in \Lambda} \ua A_i.\] 
 Conversely,  suppose \[ \bigcap_{i\in \Lambda} \ua A_i\nsubseteq \ua  (\bigcap_{i\in \Lambda} A_i),\] i.e. there exists \(x_0\in  \bigcap_{i\in \Lambda} \ua A_i\) satisfying \(x_0\notin \ua  (\bigcap_{i\in \Lambda} A_i).\)
In this case, given any \(a_0\in \bigcap_{i\in \Lambda} A_i\), since \(x_0\notin \ua  (\bigcap_{i\in \Lambda} A_i),\) then \(x_0<a_0\). For each \(i\in \Lambda\),  since \(x_0\in \bigcap_{j\in \Lambda} \ua A_j\subseteq \ua A_i\),  there exists \(a_i\in A_i\) such that  \(x_0\ge a_i\). 
Since \(a_i\le x_0< a_0\), where both \(a_i\) and \(a_0\) belong to  the convex set \(A_i\), then \(x_0\in A_i\). Hence, \(x_0\in \bigcap_{i\in \Lambda} A_i\), which is contradict to the assumption  \(x_0\notin \ua  (\bigcap_{i\in \Lambda} A_i).\) Therefore, \[ \bigcap_{i\in \Lambda} \ua A_i\subseteq \ua  (\bigcap_{i\in \Lambda} A_i),\]  thus \[ \bigcap_{i\in \Lambda} \ua A_i= \ua  (\bigcap_{i\in \Lambda} A_i).\] 

Dually, if \(\bigcap_{i\in \Lambda} A_i\ne \varnothing,\) then  \[\da (\bigcap_{i\in \Lambda} A_i)=\bigcap_{i\in \Lambda} \da A_i.\]
 \end{proof}

\begin{prop}\label{15}
Suppose that  \(\{A_i\}_{i\in \Lambda}\) and   \(\{B_i\}_{i\in \Lambda}\)  are two families of  subintervals of \([0,1]\), and  \(A_i\p B_i\) for each \(i\in \Lambda\).  If both \(\bigcup_{i\in \Lambda}{A_i}\) and \(\bigcup_{i\in \Lambda}{B_i}\) are convex, then \[\bigcup_{i\in \Lambda} A_i \p  \bigcup_{i\in \Lambda} B_i\]
 If  both \(\bigcap_{i\in \Lambda}A_i\) and  \(\bigcap_{i\in \Lambda}B_i\) are non-empty, then  \[\bigcap_{i\in \Lambda}A_i\p  \bigcap_{i\in \Lambda} B_i.\]
 \end{prop}
 
\begin{proof}
 By Lemma \ref{12}, for any \(i\in\Lambda\),  \[\ua B_i\subseteq  \ua A_i,\quad \da A_i\subseteq  \da B_i.\] Hence, \begin{equation}\label{-9} \bigcup_{i\in \Lambda}\ua B_i\subseteq \bigcup_{i\in \Lambda}\ua A_i,\quad  \bigcup_{i\in \Lambda}\da A_i\subseteq \bigcup_{i\in \Lambda}\da B_i.\end{equation}
If both \(\bigcup_{i\in \Lambda}{A_i}\) and \(\bigcup_{i\in \Lambda}{B_i}\) are convex, by (\ref{-9})  and Lemma \ref{13},   \[\ua( \bigcup_{i\in \Lambda}B_i)\subseteq \ua (\bigcup_{i\in \Lambda}A_i),\quad \da (\bigcup_{i\in \Lambda}A_i)\subseteq \da (\bigcup_{i\in \Lambda}B_i).\] This, together with  Lemma \ref{12}, implies \[\bigcup_{i\in \Lambda}A_i\p  \bigcup_{i\in \Lambda} B_i.\]
Similarly,  if both \(\bigcap_{i\in \Lambda}A_i\) and  \(\bigcap_{i\in \Lambda}B_i\) are non-empty, then \[\bigcap_{i\in \Lambda}A_i\p  \bigcap_{i\in \Lambda} B_i.\]
\end{proof}

\begin{prop}\label{17}
For all \(f,g\in \mathbf{L}\), 
 the following statements are equivalent.
\begin{itemize}
\item[(i)] \(f \sqsubseteq g.\)
\item[(ii)]  \(f^{\hat{a}}\p g^{\hat{a}}\)  for all  \(a\in(0,1)\).
\item[(iii)] \(f^{a}\p g^{a}\)  for all  \(a\in(0,1)\).
\end{itemize}
\end{prop}

\begin{proof}
\((i)\Rightarrow (ii)\). It follows from  Proposition \ref{16} immediately.

\((ii)\Rightarrow (iii)\). For any \(a\in (0,1)\), by Proposition  \ref{15},  \[f^a=\bigcap_{b\in (0,a)} f^{\hat{b}}\p\bigcap_{b\in (0,a)} g^{\hat{b}}=g^a.\]

\((iii)\Rightarrow (i)\). For any \(a\in[0,1)\), by Proposition \ref{15}, 
 \[f^{\hat{a}}=\bigcup_{b\in (a,1)} f^{b}\p\bigcup_{b\in (a,1)} g^{b}=g^{\hat{a}}.\]
By Proposition \ref{16},  \(f \sqsubseteq g\).
\end{proof}

\begin{lem} \label{18} If   \(\ast\)  is a continuous t-norm, then for \(A, B,C\in \mathcal{I}\),  \(A\p B\) implies  \(A\ast C\p B\ast C\).
\end{lem}
\begin{proof}  If \(A\p B\), then \(\ua A\subseteq  \ua B\) and \(\da B\subseteq \da A,\)
hence, \[\ua(A\ast C)=\ua A\ast \ua C\subseteq  \ua B\ast \ua C= \ua(B\ast C),\] 
\[\da(B\ast C)=\da B\ast \da C\subseteq  \da A\ast \da C= \da(A\ast C).\] Therefore, \(A\ast C\p B\ast C\). 
\end{proof}
\begin{prop}\label{.23}
Suppose that \(\ast\) is  a  continuous  t-norm    and \(\va\) is  a border continuous t-norm. Then for all  \(f_1, f_2, g\in \mathbf{L},\) \(f_1\sqsubseteq  f_2\) implies \( f_1\ast_\va g\sqsubseteq  f_2\ast_\va g\).
\end{prop}

\begin{proof}
On the one hand,   for all \(a\in (0,1)\)  and \(i=1,2\),
 \[\label{q1}(f_i\ast_\va g)^{\hat{a}}=\bigcup_{b\va c>a}
f_i^{b}\ast g^c=\bigcup_{b\va c>a, (b,c)\in (0,1)^2}f_i^{b}\ast g^c,\]
 where the second  equality follows from the fact that \(a>0\),  t-norm \(\va\) is continuous at \((1,1)\),  and  for all \(\alpha,\beta\in[0,1]\),  \(\alpha\le \beta\) implies \(f_i^{\beta}\subseteq f_i^{\alpha}\).
On the other hand,  by Proposition \ref{17},   \(f_1^{b}\p f_2^{b}\) for all \(b\in(0,1)\), then by Lemma \ref{18},  for all \(b,c\in(0,1)\), \[f_1^{b}\ast g^c\p f_2^{b}\ast g^c.\]  
From Proposition \ref{15}, for all \(a\in(0,1)\), 
\[(f_1\ast_\va g)^{\hat{a}}=\bigcup_{b\va c>a, (b,c)\in (0,1)^2}f_1^{b}\ast g^c \p\bigcup_{b\va c>a, (b,c)\in (0,1)^2}f_2^{b}\ast g^c=(f_2\ast_\va g)^{\hat{a}}.\] 
Hence, \( f_1\ast_\va g\sqsubseteq  f_2\ast_\va g\).
\end{proof}

\subsection{Associativity}

In this subsection, we  give the sufficient conditions such that \(\ast_\va\) is associative on \(\mathbf{L}\).

\begin{prop} \label{.24}
Let  \(\ast\) and \(\va\) be two t-norms. If \(\va\) is  left-continuous, then  \(\ast_\va\) is associative on \(\mathbf{M}\).
\end{prop}
\begin{proof}
Let \(f,g,h\in \mathbf{M}\). For any \(a\in  [0,1]\), 
\begin{align*}
((f\ast_\va g)\ast_\va h)(a)&=\sup\{(f\ast_\va g)(u)\va h(z)\mid  u\ast z=a\}\\
&=\sup \{ \sup\{f(x)\va g(y)\mid x\ast y=u \} \va h(z)\mid u\ast z=a\}\\
&=\sup \{ \sup\{f(x)\va g(y)\va h(z)\mid x\ast y=u \} \mid u\ast z=a\}\\
&=\sup\{f(x)\va g(y)\va h(z)\mid x\ast y\ast z=u \}.
\end{align*}
Similarly, 
\[(f\ast_\va (g\ast_\va h))(a)= \sup\{f(x)\va g(y)\va h(z)\mid x\ast y\ast z=u \}.\]
Hence,  \(\ast_\va\) is associative on \(\mathbf{M}.\)
\end{proof}

Recall that for each \(f\in \mathbf{L}\), \(f=f^L\wedge f^R\), where \(f^L\) is monotone increasing and \(f^R\) is monotone decreasing.

\begin{defn}\label{d6}
For each \(f\in \mathbf{L}\),  the subinterval  \(f^+\) of \([0,1]\) is defined by \[f^+=\cup\{[0,a]\mid f \text{ is monotone increasing on } [0,a] \}\]  
\end{defn}

\begin{lem}\label{31} For each \(f\in \mathbf{L}\),  \(x\in f^+\) if and only if \(f^R(x)=1\).\end{lem}
\begin{proof}
If \(x\in f^+\), then \(f\) is monotone increasing on \([0,x]\), thus \[f^R(x)=\bigvee f(\ua x)= \bigvee f(\ua 0)=1.\] Conversely, suppose that \(f^{R}(x)=1\).  For any \(a\le x\), since \(f^{R}\) is monotone decreasing,    \[f^R(a)\ge f^{R}(x)=1,\] then  \[f(a)=f^L(a)\wedge f^R(a)= f^L(a)\wedge 1= f^L(a),\]   thus  \[f|_{[0,x]}=f^L|_{[0,x]}.\] Therefore, \(f\) is monotone increasing on \([0,x]\), i.e.  \(x\in f^+\).   
\end{proof}

\begin{prop} \label{32}
If \(\va\)  is a border-continuous t-norm, then  \(\wedge_\va\) is associative on \(\mathbf{L}\).
\end{prop}
\begin{proof}
Claim 1, for all \(f,g\in \mathbf{L},\)
\[(f\wedge_\va g)(x)=\begin{cases} f(x)\vee g(x), & x\in f^+\cap g^+,\\ f(x),& x\in g^+\setminus f^+,\\  g(x), & x\in f^+\setminus g^+, \\ f(x)\va g(x), & x\notin f^+\cup g^+.\end{cases}\] 
It is easy to see 
 \[(f\wedge_\va g)(x)
 =\bigvee_{a\wedge b=x} f(a)\va g(b)
 =\max\{\bigvee (f(x)\va g(\ua x)), \bigvee (f(\ua x)\va g(x))\}.\]
For all \(x\in[0,1]\),  let \[F_1(x)=\bigvee  (f(x)\va g(\ua x)),\]  \[F_2(x)=\bigvee ( f(\ua x)\va g(x)),\] then
\begin{equation}\label{.10}(f\wedge_\va g)(x)=F_1(x)\vee F_2(x).\end{equation} 
In the following, we shall show that   \begin{equation}\label{.11}F_1(x)=\begin{cases} f(x), & x \in g^+,\\ f(x)\va g(x),& x \notin g^+, \end{cases}\end{equation}
\[\label{.12}F_2(x)=\begin{cases} g(x), & x \in f^+,\\ f(x)\va g(x),& x \notin f^+. \end{cases}\] These, together with (\ref{.10}), imply that claim 1  holds.

If  \(x\in g^+\), by Lemma \ref{31},
 \[g^{R}(x)=\bigvee g(\ua x)=1.\]  Since \(\va\) is border continuous, 
 
  \[F_1(x)=\bigvee (f(x)\va g(\ua x))= f( x)\va \bigvee g( \ua x)=f(x).\]
 If \(x\notin g^+\),  since \(g\) is monotone decreasing on \([x,1]\),  \[\max g(\ua x)=g(x),\] thus  \[F_1(x)=\bigvee (f(x)\va g(\ua x))=f(x)\va g(x).\] Hence, (\ref{.11}) holds. The proof of  \(F_2\) is similar.

 Claim 2, for  \(f,g\in \mathbf{L}\),   \((f\wedge_\va g)^+=f^+\cap g^+\). Since both intervals \(f^+\) and \(g^+\) contain the point \(0\), suppose  \[f^+ \subseteq g^+\] without  loss of the generality. 
By claim 1,  \[ (f\wedge_\va g)(x)=f(x)\vee g(x),  \quad \text{for } x\in f^+\cap g^+.\] This, together with that both   \(f\) and \(g\) are monotone increasing on \(f^+\cap g^+\), implies that    \( f\wedge_\va g\) is monotone increasing on \(f^+\cap g^+\). Hence, \[f^+\cap g^+\subseteq  (f\wedge_\va g)^+.\]
Conversely,  in order to show  \[ (f\wedge_\va g)^+ \subseteq f^+\cap g^+=  f^+,\] given any \(x_0\in[0,1] \setminus f^+\),  we  shall show that \(x_0\notin (f\wedge_\va g)^+\).  Hence, claim 2 holds. 

In fact, for each  \(x\in[0,1] \setminus f^+\), since \(f^+ \subseteq g^+\), then  either \(x\in  g^+\setminus f^+\)  or \(x\notin  g^+=f^+\cup g^+.\)  In each case, we always have   \[(f\wedge_\va g)(x)\le f(x)\] from claim 1. This, together with the fact \(\ua x_0\subseteq [0,1] \setminus f^+\), implies   \[\bigvee (f\wedge_\va g)(\ua x_0)\le \bigvee f(\ua x_0).\] Due to \(x_0\notin f^+\),
then \( f^R(x_0)<1\), hence  \[(f\wedge_\va g)^R(x_0)=\bigvee (f\wedge_\va g)(\ua x_0)\le \bigvee f(\ua x_0)=f^R(x_0)<1,\] which implies \[x_0\notin (f\wedge_\va g)^+\] 
as desired.

For all \(f,g,h\in \mathbf{L}\), by claim 1 and claim 2, one easily  obtains the following results.

If  \(x\in f^+\cap g^+\cap h^+\),  then
   \[((f\wedge_\va g)\wedge_\va h)(x)=(f\wedge_\va g)(x)\vee h(x)= f(x)\vee g(x)\vee h(x),\]  \[(f\wedge_\va (g\wedge_\va h))(x)=f(x)\vee (g\wedge_\va h)(x)= f(x)\vee g(x)\vee h(x).\]

If  \(x\in (g^+\cap h^+)\setminus f^+,\) then \[((f\wedge_\va g)\wedge_\va h)(x)=(f\wedge_\va g)(x)=f(x),\] 
\[(f\wedge_\va (g\wedge_\va h))(x)=f(x).\]

If \(x\in f^+\setminus(g^+\cup h^+),\) then \[((f\wedge_\va g)\wedge_\va h)(x)=(f\wedge_\va g)(x)\va h(x)=g(x)\va h(x),\] \[ (f\wedge_\va (g\wedge_\va h))(x)=(g\wedge_\va h)(x)=g(x)\va h(x).\]

If  \(x\notin f^+\cup g^+\cup h^+,\)  then \[((f\wedge_\va g)\wedge_\va h)(x)=(f\wedge_\va g)(x)\va h(x)=f(x)\va g(x)\va h(x),\] \[ (f\wedge_\va (g\wedge_\va h))(x)=f(x)\va (g\wedge_\va h)(x)=f(x)\va g(x)\va h(x).\] 

Similarly, in other cases,  the following equality  still holds,
 \[((f\wedge_\va g)\wedge_\va h)(x)=(f\wedge_\va (g\wedge_\va h))(x).\]
Therefore, \(\wedge_\va\) is associative on \(\mathbf{L}\).
\end{proof}

\subsection{Necessary and sufficient conditions}

In this subsection, we give the necessary and sufficient conditions under which \(\ast_\va\) is a t-norm on \((\mathbf{L},\sq)\). 

\begin{prop}\cite{hu141}\label{.17}
Let \(\ast\) and \(\va\) be two t-norms. Then  convolution  \(\ast_\va: \mathbf{M}^2\to \mathbf{M}\) is commutative and  \(\bar{1}\) is  the unit.
\end{prop}

\begin{thm}\label{22}
The convolution \(\ast_\va\) is a t-norm on \((\mathbf{L},\sq)\) if and only if all the following conditions hold.
\begin{itemize}
\item[(i)] \(\ast\)  is a  continuous t-norm.
\item[(ii)] If \(\ast\ne\wedge\),  then  \(\va\) is a left-continuous t-norm. 
\item[(iii)] If  \(\ast=\wedge\), then \(\va\) is a  border continuous t-norm. 
\end{itemize}
\end{thm}
\begin{proof}
(\(\Rightarrow\))  It is clear from Proposition \ref{.21}.
\\
(\(\Leftarrow\)) By Proposition \ref{.19}, \(\ast_\va\) is closed on \(\mathbf{L}\). By Proposition \ref{.17}, \(\ast_\va\) is commutative and has the unit \(\bar{1}\).  By Proposition \ref{.23}, \(\ast_\va\) is monotone increasing in each component. If \(\ast\ne\wedge\), then \(\va\) is left-continuous, and thus \(\ast_\va\)  is associative by Proposition \ref{.24}. 
If \(\ast=\wedge\), then \(\ast_\va=\wedge_\va\)  is associative by Proposition \ref{32}.
 Therefore, \(\ast_\va\) is a t-norm on \((\mathbf{L},\sq)\).
\end{proof}

\begin{cor} 
If \(\ast\) is a continuous t-norm and \(\va\) is a left-continuous t-norm, then \(\ast_\va\) is a  t-norm on \((\mathbf{L},\sq)\).
\end{cor}

In general, the computation of convolution is very complex. Here, we provide one kind of convolutions which are t-norms on \((\mathbf{L},\sq)\) and are easily to calculate.

\begin{prop} \label{p}
Let \(\va\)  be a border-continuous t-norm. Then \(\wedge_\va\)  is a t-norm on \((\mathbf{L}, \sq )\), and for all \(f,g\in\mathbf{L}\), 
\[(f\wedge_\va g)(x)=\begin{cases} f(x)\vee g(x), & x\in f^+\cap g^+,\\ f(x),& x\in g^+\setminus f^+,\\  g(x), & x\in f^+\setminus g^+, \\ f(x)\va g(x), & x\notin f^+\cup g^+.\end{cases}\] 
\end{prop}
\begin{proof} It is straightforward from Theorem \ref{22} and the proof of Proposition \ref{32}.
\end{proof}
\section{Triangular conorms}

In this section, we  give the necessary and sufficient conditions under which \(\ast_\va\) is a t-conorm or \(t_r\)-conorm on \((\mathbf{L},\sq)\).

Define \(N:[0,1]\to [0,1],\)  \[N(x)=1-x.\] \(\neg: \mathbf{M}\to \mathbf{M}\)   \[\neg f=f\circ N.\]
 \(\bar{\ast}: [0,1]^2\to [0,1]\), \[\bar{\ast}(x,y)=N(N(x)\ast N(y)).\]
Obviously, \(\neg(\neg f)=f\), \(\bar{\bar{\ast}}=\ast\), and \(\ast\) is a t-norm if and only if  \(\bar{\ast}\) is a t-conorm.
\begin{rem}\label{0.1} 
Since \(N\) is an isomorphism between the algebras \(([0,1],\ast)\) and  \(([0,1],\bar{\ast})\), i.e.  for all \(x,y\in[0,1]\), \[N(x\ast y)=N(x)\bar{\ast}N(y),\] then 
for any \(A, B\subseteq [0,1],\)  \[N(A\ast B)=N(A)\bar{\ast} N(B).\]
\end{rem}
\begin{lem}\label{.16}
For any \(f\in \mathbf{M}\) and \(a\in [0,1],\) \[(\neg f )^a=N(f^a), \quad (\neg f )^{\hat{a}}=N(f^{\hat{a}}).\]
\end{lem}
\begin{proof}
\(x\in (\neg f)^a\)
 \(\Leftrightarrow  (\neg f)(x)\ge a\)
  \(\Leftrightarrow  f(N(x))\ge a\)
   \(\Leftrightarrow  N(x)\in  f^a\) \(\Leftrightarrow  N(N(x))\in  N(f^a)\)
      \(\Leftrightarrow  x\in  N(f^a).\)   
 Similarly, \((\neg f )^{\hat{a}}=N(f^{\hat{a}})\).
  \end{proof}

\begin{prop}\label{..24}

If \(\va\) is monotone increasing in each place, then
  \(\neg: (\mathbf{L}, \sq)\to (\mathbf{L}, \sq)\)  is an order-reversing isomorphism such that for all \(f,g\in \mathbf{L}\),  \[\neg (f\ast_\va g)=(\neg f)\bar{\ast}_\va (\neg g).\]  
  \end{prop}
 \begin{proof} 
 Firstly, for any  binary function \(\ast\) on \([0,1]\) and  \(a\in [0,1)\), 
  \begin{align*}
(\neg(f\ast_\va g))^{\hat{a}}
 &=N((f\ast_\va g)^{\hat{a}})~\text{(by Lemma \ref{.16})}
 \\
    &= N(\bigcup_{b\va c>a} f^b\ast g^c) (\text{by Proposition \ref{.14}})
 \\
  &= \bigcup_{b\va c>a} N(f^b\ast g^c)  
 \\
    &=\bigcup_{b\va c>a} N(f^b)\bar{\ast} N(g^c)~(\text{by Remark \ref{0.1}})
  \\
      &=\bigcup_{b\va c>a} (\neg f)^b\bar{\ast} (\neg g)^c
  \\ &=((\neg f)\bar{\ast}_\va (\neg g))^{\hat{a}},
 \end{align*}
  hence, \begin{equation}\label{l2}\neg (f\ast_\va g)=(\neg f)\bar{\ast}_\va (\neg g).\end{equation} 
 
 Next, 
 we show that  \(\neg\)  is an order-reversing isomorphism.  Obviously, \(\neg:\mathbf{L}\to \mathbf{L}\) is a bijection. For  any \(f,g\in \mathbf{L}\), if  \(f\sqsubseteq g\), then \(f \vee_\wedge g=g.\) By  (\ref{l2}),  \[\neg g=\neg (f \vee_\wedge g)=(\neg f)\bar{\vee}_\wedge (\neg g)=(\neg f)\wedge_\wedge (\neg g),\] which means \(\neg g\sqsubseteq \neg f.\) Hence,  \(\neg\)  is an order-reversing isomorphism.
 \end{proof}

  \begin{prop}\label{8.1}
  If \(\ast_\va\) is  a   t-conorm on \((\mathbf{L}_\mathbf{u},\sq)\) or \((\mathbf{L},\sq)\),  then 
\begin{itemize}
\item[(i)] \(\ast\)  is a  continuous t-conorm and \(\va\) is a t-norm.
\item[(ii)] \(\ast_\va\) is  closed on  \(\mathbf{J}\) and \(\mathbf{J}^{[2]}\), respectively.
\item[(iii)]  \(\overline{[0, 1]}\ast_\va \overline{[a,b]}=\overline{[a,b]}.\)
\end{itemize}
 \end{prop}
 \begin{proof}
 Similar to the proof of  Proposition \ref{8}.   
\end{proof}

 \begin{prop}\label{21}
  \(\ast_\va\) is a  t-norm on \((\mathbf{L},\sq)\) if and only if  \(\bar{\ast}_\va\) is a t-conorm on \((\mathbf{L},\sq).\)
 \end{prop}
 \begin{proof}
 
If   \(\ast_\va\) is a  t-norm,  then  \(\va\) is a t-norm from Proposition  \ref{8}.   By Proposition \ref{..2} and Proposition \ref{..24},  \(\bar{\ast}_\va\) is a t-conorm.  Another direction is similar.  
 \end{proof}

 \begin{thm} \label{23}
The convolution \(\ast_\va\) is a t-conorm on \((\mathbf{L},\sq)\) if and only if all the following conditions hold.
\begin{itemize}
\item[(i)] \(\ast\)  is a  continuous t-conorm.
\item[(ii)] If \(\ast\ne\vee\), then \(\va\) is a left-continuous t-norm. 
\item[(iii)] If \(\ast=\vee\), then \(\va\) is a  border continuous t-norm. 
\end{itemize}
\end{thm}
 \begin{proof}
 It follows from Theorem \ref{22} and Proposition  \ref{21} immediately.
 \end{proof}
 
 \begin{thm}  
The convolution \(\ast_\va\) is a  \(t_r\)-conorm on \((\mathbf{L},\sq)\)   if and only if    \(\ast_\va\) is a t-conorm on \((\mathbf{L},\sq).\)
\end{thm}
\begin{proof}
It is straightforward from Proposition \ref{8.1}.
\end{proof}

 \section{Convolution}
 In this study, we have investigated the construction of t-norms on \((\mathbf{L},\sq)\) for T2 RFSs. The main results are as follows:
 \begin{itemize}
 \item[1)]  Under the condition that both \(\ast\) and \(\va\) are two binary operation on [0,1], and   \(\va\) is surjective, we have  provided the necessary and sufficient conditions under which \(\ast_\va\) is a t-norm (\(t_r\)-norm) on \((\mathbf{L},\sq)\), see Theorem \ref{24} and Theorem \ref{22}.  Dual results of  t-conorms and \(t_r\)-conorms have been obtained.  
 \item[2)] When \(\va\) is a border-continuous t-norm, the calculation of t-norm \(\wedge_\va\) can be expressed simply,  see Proposition \ref{p}. 
 \item[3)]  We have provided the equivalent characterization  of convolution  order \(\sq\) on \(\mathbf{L}\) through the \(\alpha\)-cut and the strong \(\alpha\)-cut, respectively,  see Proposition \ref{17}. 
  \end{itemize}
  For the further work, we will concentrate  on   characterizations of  t-norms on \((\mathbf{L_u},  \sq)\). 

\section*{Acknowledgement}
The author acknowledges the support of the Southwest Minzu University Research  Startup Funds  (No:RQD2024011).

\vfill
\end{document}